\documentclass[11pt, a4paper,pagebackref]{amsart}
\usepackage{amsmath}
\usepackage{amssymb}
\usepackage{amsthm}
\usepackage{amscd}
\usepackage{tikz-cd}
\usepackage{color}
\usepackage{hyperref}

\usepackage[inner=2.3cm,outer=2.3cm,bottom=3.3cm]{geometry}
\usepackage{setspace}

\usepackage{stmaryrd}
\numberwithin{equation}{section}

\usepackage[colorinlistoftodos]{todonotes}

\usepackage[all]{xy}
\usepackage{enumerate}
\usepackage{hyperref}

\usepackage{color}
\hypersetup{colorlinks=true} 

\newtheorem{theorem}{Theorem}[section]

\newtheorem{lemma}[theorem]{Lemma}
\newtheorem{proposition}[theorem]{Proposition}
\newtheorem{corollary}[theorem]{Corollary}

\theoremstyle{definition}
\newtheorem{definition}[theorem]{Definition}
\newtheorem{newclaim}[theorem]{}

\newtheorem{setup}[theorem]{Setup}
\newtheorem{remark}[theorem]{Remark}
\newtheorem{remark and definition}[theorem]{Remark and Definition}
\newtheorem{remark and notation}[theorem]{Remark and Notation}

\newtheorem*{funding}{Funding}
\newtheorem*{agra}{Acknowlegment}

\newtheorem{example}[theorem]{Example}

\newtheorem{question}[theorem]{Question}

\newcommand\Spec{\operatorname{Spec}}
\newcommand\Hom{\operatorname{Hom}}
\newcommand\Ext{\operatorname{Ext}}
\newcommand\Tor{\operatorname{Tor}}

\newcommand\depth{\operatorname{depth}}

\newcommand\Ker{\operatorname{\Ker}}

\newcommand\pd{\operatorname{pd}}
\newcommand{\ldt}{\otimes^\mathbf{L}}
\newcommand\id{\operatorname{id}}

\newcommand\RHom{\operatorname{\mathbf{R}Hom}}

\DeclareMathOperator{\Hdim}{\mathcal{H}-dim}

\DeclareMathOperator{\gdim}{G-dim}

\author[Ajani]{Tatheer Ajani}
\address{University of Texas at Arlington, Arlington, TX 76019, U.S.A.}
\email{tatheer.ajani@gmail.com}
\urladdr{https://sites.google.com/view/tatheerajani}

\author[Martins]{Paulo Martins}
\address{Universidade de S{\~a}o Paulo -
ICMC, Caixa Postal 668, 13560-970, S{\~a}o Carlos-SP, Brazil}
\email{paulomartinsmtm@gmail.com}
\urladdr{https://sites.google.com/view/martinsp}

\author[Mendoza-Rubio]{Victor D. Mendoza Rubio}
\address{Universidade de S{\~a}o Paulo -
ICMC, Caixa Postal 668, 13560-970, S{\~a}o Carlos-SP, Brazil}
\email{vicdamenru@gmail.com}
\urladdr{https://sites.google.com/view/vicdamenru}

\author[Soto Levins]{Andrew J. Soto Levins}
\address{Texas Tech University, TX 79409. U.S.A.}
\email{ansotole@ttu.edu}
\urladdr{https://sites.google.com/view/andrewjsotolevins}


\begin{document}

\title{Generalized symmetry in the vanishing of Ext}

\begin{abstract}
Foundational work by Avramov and Buchweitz, as well as by Huneke and Jorgensen, established symmetry results for the vanishing of Ext over local complete intersections and AB-rings, respectively. Later, several authors studied symmetry in the vanishing of Ext over local Gorenstein rings under additional homological assumptions on the modules involved. In this work, we provide a generalized symmetry in the vanishing of Ext for homologically finite complexes over a Noetherian ring with a dualizing complex, under additional homological conditions on the complexes involved. The results in this paper extend and unify several symmetry results on the vanishing of Ext previously proved over local Gorenstein rings and show that the dualizing complex has been hidden in these results.
\end{abstract}

\maketitle

\section{Introduction}
Throughout this paper, $R$ denotes a commutative Noetherian ring. The relation between the vanishing of all higher $\Ext_R^i(M,N)$ and $\Ext_R^i(N,M)$ has been extensively studied when $R$ is Gorenstein \cite{avramov,BERGH,CH,GJT,JorgensenHuneke,JO-SE,PE-JO,NASSEH-SYMMETRY}. Utilizing support varieties, Avramov and Buchweitz \cite[Theorem III]{avramov} established a foundational result in this context: they proved that if $M$ and $N$ are finitely generated $R$-modules over a local complete intersection $R$, then the following equivalence holds:\begin{align}\label{SYM}\Ext_R^i(M,N)=0 \text{ for all }i\gg 0 \Longleftrightarrow \Ext_R^i(N,M )=0 \text{ for all } i  \gg 0.\end{align}
Following this, Avramov and Buchweitz raised the question of determining which class of rings satisfies this equivalence for all finitely generated modules $M$ and $N$. It is easy to verify that such rings must be Gorenstein. Huneke and Jorgensen \cite{JorgensenHuneke} made a substantial advancement by introducing AB rings, demonstrating that the equivalence in (\ref{SYM}) holds for all finitely generated modules over AB rings, and that the class of AB rings strictly contains that of complete intersections. To give the definition of an AB ring, we first recall the so-called Auslander conditions from \cite{JorgensenHuneke}, which were originally motivated by a conjecture of Auslander \cite[p. 815]{AC1}:

\begin{itemize}
    \item[(AC)] For every finitely generated $R$-module $M$, there exists an integer $b_M \geq 0$ such that for every finitely generated $R$-module $N$, if $\Ext_R^i(M,N)=0$ for all $i \gg 0$, then $\Ext_R^i(M,N)=0$ for all $i > b_M$.
    \item[(UAC)] There exists an integer $b \geq 0 $ such that for all finitely generated $R$-modules $M$ and $N$, if $\Ext_R^i(M,N)=0$ for all $i \gg 0$, then $\Ext_R^i(M,N)=0$ for all $i > b$.
\end{itemize}
An AB ring is defined as a local Gorenstein ring satisfying (UAC). Jorgensen and Şega \cite{JO-LI, JO-SE} demonstrated the existence of a local Gorenstein ring that is not AB and showed that the equivalence in (\ref{SYM}) does not always hold over a local Gorenstein ring. Following these results, the symmetry in the vanishing of $\operatorname{Ext}$ continued to be investigated over Gorenstein local rings by imposing additional homological conditions on the finitely generated modules $M$ or $N$. For instance, when $R$ is a Gorenstein local ring, the cases where the equivalence (\ref{SYM}) has been established include: when $M$ and $N$ have finite complete intersection dimension, by Jørgensen \cite[Theorem 4.1]{PE-JO}; when $M$ or $N$ has finite quasi-projective dimension, due to recent work by Gheibi, Jorgensen, and Takahashi \cite[Theorem 6.16]{GJT}; and when $M$ and $N$ have reducible complexity, as shown by Bergh \cite{BERGH}. Nasseh and Tousi \cite{NASSEH-SYMMETRY} investigated symmetry results over complete intersection local rings, only considering that $M$ is finitely generated. Additionally, Christensen and Holm \cite[Theorem B.2]{CH} extended the result of Huneke and Jorgensen to homologically finite complexes over rings that are not necessarily local.

Motivated by these results, the main goal of this paper is to investigate a generalized symmetry in the vanishing of $\operatorname{Ext}$ for homologically finite complexes over a commutative Noetherian ring, dropping the assumption that $R$ is Gorenstein. The results in this paper extend and unify several symmetry results proven before and show that the dualizing complex has been hidden in these results. This work, in turn, suggests a new approach to study symmetry in the vanishing of Ext via the dualizing complex. As the main tool to prove our results, we extended the definition of AB-dimension given by Araya \cite{araya} for finitely generated $R$-modules over a local ring to the context of homologically finite complexes over a Noetherian ring, using the definition of Auslander bounds for complexes given by Soto Levins in \cite{levins}. The definition of AB-dimension for complexes is given in Definition \ref{def:AB-dim}. We now state the main results of this paper.

\begin{theorem}[See Theorem \ref{main}]
 Let $R$ be a Noetherian ring with a dualizing complex $D$, and let $M$ and $N$ be  complexes in $D_b^f(R)$ with 
$\operatorname{AB-dim}_R(M)< \infty$ and $\operatorname{AB-dim}_R( \RHom_R(M,R))< \infty$. Then 
\[\Ext_R^i(M,N)=0 \text{ for all }i\gg 0 \Longleftrightarrow \Ext_R^i(N,M \ldt_R D)=0 \text{ for all } i  \gg 0.\]
\end{theorem}
\begin{corollary}[See Corollaries \ref{corollary}, \ref{cor:qpd} and \ref{cor:hdim}]
    Let    $R$ be a Noetherian ring with a dualizing complex $D$, and let $M$ and $N$ be complexes in $D_b^f(R)$. Then 
\[\Ext_R^i(M,N)=0 \text{ for all }i\gg 0 \Longleftrightarrow \Ext_R^i(N,M \ldt_R D)=0 \text{ for all } i  \gg 0\]
holds under any of the following hypotheses
\begin{enumerate}
    \item $R$ satisfies \textnormal{(AC)} and $\gdim_R(M)< \infty$.
    \item $\operatorname{CI-dim}_R(M)< \infty$.
    \item $M$ is a finitely generated $R$-module, $\operatorname{qpd}_R(M)< \infty$ and $\gdim_R(M)< \infty$.
\end{enumerate}
\end{corollary}
The results listed above combine and strengthen several symmetry results on the vanishing of Ext over Gorenstein rings, extending them to homologically finite complexes over Noetherian rings that admit a dualizing complex. As an application of our main results, we establish new cases in which the symmetry in the vanishing of $\Ext$ (\ref{SYM}) holds without the assumption that $R$ is Gorenstein.

\begin{corollary}[See Corollary \ref{cor:symmetrygpd}]
 Let $R$ be a Noetherian ring with a dualizing complex, and let $M$ and $N$ be complexes in $D_b^f(R)$ with $\operatorname{Gpd}_R(\RHom_R(N,M))< \infty$. Then
\begin{align*}
\Ext_R^i(M,N)=0 \text{ for all }i\gg 0 \Longleftrightarrow \Ext_R^i(N,M )=0 \text{ for all } i  \gg 0 
\end{align*}
holds under any of the following hypotheses
\begin{enumerate}
    \item $R$ satisfies (AC) and $\gdim_R(M)< \infty$.
    \item $\operatorname{CI-dim}_R(M)< \infty$.
    \item $M$ is a finitely generated $R$-module, $\operatorname{qpd}_R(M)< \infty$ and $\gdim_R(M)< \infty$.
\end{enumerate}
\end{corollary}

We now describe the contents of the article. In Section \ref{section:preliminaries}, we provide some notations and definitions that are used in the article. In Section \ref{section:AB-dim}, we develop the principal tool used throughout this paper, namely, the theory of AB-dimension for homologically finite complexes over Noetherian rings. Section \ref{section:main} is devoted to the proof of our main theorem (Theorem \ref{main}), which establishes the first Generalized Symmetry in the vanishing of Ext result. In Section \ref{section:corollaries}, we demonstrate how Theorem \ref{main} unifies and extends several known symmetry results for the vanishing of Ext over Gorenstein rings.

\section{Preliminaries}\label{section:preliminaries}

In this section, we introduce fundamental definitions that will be used throughout the paper.







\begin{newclaim}[\textbf{Auslander bound and AB-dimension for modules}] 
\begin{definition}
For finitely generated $R$-modules $M$ and $N$, define 
\begin{enumerate}
    \item $\operatorname{P}_R(M,N)  = \sup \lbrace i \mid \operatorname{Ext}_R^i(M,N) \neq 0 \rbrace$. 
    \item $b_M^R = \sup \lbrace \operatorname{P}_R(M,N) \mid N \text{ is finitely generated and } \operatorname{P}_R(M,N)< \infty  \rbrace $.
    \item \cite[Definition 2.2(4)]{araya} If $R$ is local, then $\operatorname{AB-dim}_R(M) = \sup  \lbrace \gdim_R(M), b_M^R \rbrace$.
\end{enumerate}
The invariants $b_M^R$ and $\operatorname{AB-dim}_R(M)$ are called the \textit{Auslander bound} and the \textit{AB-dimension} of $M$.
\end{definition}

Note that $\operatorname{P}_R(M,N)< \infty$ if and only if $\Ext_R^i(M,N)=0$ for all $i \gg 0$. In this case, it is easy to see that $b_M^R \leq \pd_R(M)$. When $\operatorname{pd}_R (M) < \infty$, then we have $b_M^R = \operatorname{pd}_R (M)$ by \cite[Theorem 2.11.1]{WJ}.

A local ring $(R,\mathfrak{m},k)$ is an AB ring if and only if all finitely generated $R$-modules are of finite AB-dimension (see \cite[Theorem 1.2]{araya}). Moreover, it is easy to see that $\operatorname{AB-dim}_R(k) < \infty$ if and only if $R$ is Gorenstein. 
\end{newclaim}

\begin{newclaim}[\textbf{Derived category backgound}]  In this article, all complexes are indexed homologically, so if $M$ is a complex, then $M$ has the form 
\[M = \cdots \rightarrow M_{i+1} \xrightarrow{\partial_{i+1}} M_i \xrightarrow{\partial_i} M_{i-1} \rightarrow \cdots.\]
The supremum and infimum of $M$ are defined as:
\[\sup(M) = \sup \lbrace i \mid H_i(M) \neq 0 \rbrace \text{ and } \inf(M) = \inf \lbrace i \mid H_i(M) \neq 0 \rbrace.\]

The cokernels and kernels of differentials of $M$ are denoted by $C_i(M)= \operatorname{coker} \partial^M_{i+1}$ and $Z_i(M) = \ker \partial_i^M$ for each $i$. A shift of $M$ is a complex $\Sigma^n M$ with $(\Sigma^n M)_i = M_{i-n}$ and differential $\partial_i^{\Sigma^n M} = (-1)^n \partial_{i-n}^M$. The truncations of $M$ are
\[M_{\leq n} = 0 \rightarrow M_n \rightarrow M_{n-1} \rightarrow \cdots, \quad M_{\geq n} = \cdots \rightarrow M_{n+1} \rightarrow M_n \rightarrow0\]
and
\[M_{\supseteq n} = \cdots \rightarrow M_{n+2} \rightarrow M_{n+1} \rightarrow Z_n(M) \rightarrow0.\]
Throughout this paper, we work in the derived category of the ring $R$, denoted $D(R)$. This category is obtained from the category of chain complexes by formally inverting all quasi-isomorphisms. Unlike the category of chain complexes, the derived category is rarely abelian. Instead, the derived category is a triangulated category, where exact triangles take the place of short exact sequences. We let $D_b(R)$ denote the full subcategory of $D(R)$ consisting of all complexes isomorphic to a bounded complex in $D(R)$ and let $D^f(R)$ denote the subcategory of $D(R)$ whose complexes have degreewise finitely generated homology. Moreover, the full subcategory $D_b(R) \cap D^f(R)$ is denoted by $D_b^f (R)$. Complexes in $D_b(R)$ and $D_b^f(R)$ are called \textit{homologically bounded} and \textit{homologically finite}, respectively. We say that a complex is \textit{homologically bounded to the right} if it is isomorphic in $D(R)$ to a complex that is bounded to the right. 

The classical functors $\Hom_R(M,-)$, $\Hom_R(-,M)$ and $M \otimes_R -$ defined in the category of $R$-complexes induce derived functors $\RHom_R(M,-)$, $\RHom_R(-,M)$ and $M \ldt_R -$ which operate on objects and morphisms in $D(R)$. These derived functors are computed using semi-projective, semi-injective, and semi-flat resolutions, which exist for any complex. Explicitly,
\begin{align*}
\RHom_R(M,N) & = \Hom_R(P,N) \text{ (where $P$ is a semi-projective resolution of $M$)},
\\
 \RHom_R(N,M) & = \Hom_R(N,I) \text{  (where $I$ is a semi-injective resolution of $M$),} \\ 
M \ldt_R N  & = F \otimes_ R N \text{ (where $F$ is a semi-flat resolution of $M$)}.
\end{align*}
As for modules, one set: 
\begin{align*}
    \Ext_R^i(M,N) = H_{-i} (\RHom_R(M,N)) \text{ and } \Tor_i^R(M,N) = H_i(M \ldt_R N).
\end{align*}
For more details, the authors recommend the reference \cite{LarsBook}.
\end{newclaim}
\begin{newclaim}[\textbf{Gorenstein dimensions}] Let $M \in D(R)$ be a complex.
\begin{enumerate}
    \item An exact complex $T$ of projective modules is \textit{totally acyclic} if $\Hom_R(T,Q)$ is exact for all projective modules $Q$, and a \textit{complete projective resolution} of $M$ is a diagram 
    \[T \xrightarrow{\tau} P \xrightarrow{\pi} M\]
    where $\pi$ is a semi-projective resolution, $T$ is totally acyclic, and $\tau_i$ is an isomorphism for $i \gg0$. Following \cite{veliche}, the \textit{Gorenstein projective dimension} of $M$, denoted by $\operatorname{Gpd}_R(M)$, is the least integer $n$ so that there exists a complete projective resolution with $\tau_ i$ an isomorphism for all $i \geq n$. 
\item  When $M \in D_b^f(R)$, we simply say \textit{Gorenstein dimension} and use the simpler notation $\gdim_R(M):=\operatorname{Gpd}_R(M)$. In this case, it is well known that $\gdim_R(M)< \infty$ if and only if $\RHom_R(M,R) \in D_b^f(R)$ and the canonical map 
$M \rightarrow \RHom_R(\RHom_R(M,R),R)$
is a quasi-isomorphism. When $\gdim_R(M)<\infty$, we have $\gdim_R(M) = - \inf(\RHom_R(M,R))$.
\end{enumerate}
\end{newclaim}

\begin{newclaim}[\textbf{Tate (co)homology}]
Let $M$ and $N$ be complexes in $D_b^f(R)$. If $T \rightarrow P \rightarrow M$ is a complete projective resolution of $M$, then the \textit{n-th Tate homology} module of $M$ against $N$ is
\[\widehat{\operatorname{Tor}}_n^R(M,N) = H_n (T \otimes_R N)\]
and the \textit{n-th Tate cohomology} module of $M$ against $N$ is 
\[\widehat{\operatorname{Ext}}_R^n(M,N) = H_{-n}(\Hom_R(T,N)).\]
For more details on Tate (co)homology, we recommend the references \cite{LarsDave} and \cite{veliche}.
\end{newclaim}

\begin{newclaim}[\textbf{Auslander bound for complexes}]
The Auslander bound for complexes was introduced by Soto Levins in \cite{levins}.
For $M$ and $N$ complexes with $\inf(N)$ an integer, define 
\begin{align*}
    \operatorname{P}_R(M,N) & = \sup \lbrace n \in \mathbb{Z} \mid \operatorname{Ext}_R^{n- \inf(N)} (M,N) \neq 0 \rbrace \\ & = \inf(N) - \inf(\RHom (M,N)).
\end{align*}
For $M \in D_b^f(R)$, the \textit{Auslander bound} of $M$ is: 
\begin{align*} 
B_M^R = \sup \lbrace \operatorname{P}_R(M,N) \mid N \in D_b^f(R) \text{ and } \operatorname{P}_R(M,N) < \infty \rbrace.
\end{align*}
\end{newclaim}
\begin{newclaim}[\textbf{Complete intersection dimension}] Sather-Wagstaff studied complete intersection dimension for complexes in \cite{sather}.  Let $M$ be a complex in $D_b^f(R)$. When $R$ is local, a \textit{(codimension $c$) quasi-deformation} of $R$ is a diagram of local homomorphisms $R \rightarrow R' \leftarrow Q$ such that the first map is flat and the second map is surjective with kernel generated by a $Q$-regular sequence (of length $c$). In this situation, the \textit{complete intersection dimension} of $M$ is: 
\begin{align*}
 \operatorname{CI-dim}_R(M) = \inf \lbrace \pd_Q(M \ldt_R R') - \pd_Q(R') \mid R \rightarrow R' \leftarrow Q \text{ is a quasi-deformation} \rbrace .
\end{align*}
When $R$ is not necessarily local, the \textit{complete intersection dimension} of $M$ is
\begin{align*}
\operatorname{CI-dim}_R(M) = \sup \lbrace \operatorname{CI-dim}_{R_{\mathfrak{m}}}(M_{\mathfrak{m} }) \mid \mathfrak{m} \in \operatorname{Max}(R) \rbrace,
\end{align*}
where $\operatorname{Max}(R)$ is the set of all maximal ideals of $R$.
\end{newclaim}

\section{AB-dimension for complexes}\label{section:AB-dim}
In this section, we define AB-dimension for homologically finite complexes over a Noetherian ring, extending the framework studied by Araya for finitely generated modules over a local ring in \cite{araya}. One goal of this section is to prepare auxiliary results that are needed to prove our main theorem and its consequences.

\begin{definition}\label{def:AB-dim}
Let $M \in D_b^f(R)$ be a complex. The $\operatorname{AB}$\textit{-dimension} of $M$ is \[\operatorname{AB-dim}_R(M) = \sup \lbrace B_M^R, \operatorname{G-dim}_R(M) \rbrace.\]
\end{definition}

\begin{remark}
(1) If $M$ is a finitely generated $R$-module and $\gdim_R(M) < \infty$, then $b_M^R =B_M^R$ by \cite[Lemma 3.5 and Corollary 4.3]{levins}. Thus $\operatorname{AB-dim}_R(M)= \sup  \lbrace \gdim_R(M),b_M^R \rbrace $ for all finitely generated $R$-module $M$. In particular, if $M$ is a module over a local ring $R$, then the above definition of AB-dimension agrees with Araya's definition.

(2) Many results considering a homologically finite complex \(M\) with 
\(\operatorname{G\text{-}dim}_R(M)<\infty\) and \(B_M^R<\infty\) were established in \cite{levins}. 
Using the definition of AB-dimension for homologically finite complexes introduced above, 
these results yield versions of the Ischebeck's formula 
(\cite[Theorem 4.5]{levins}), the Auslander-Reiten conjecture 
(\cite[Theorem 5.2]{levins}), and the derived depth formula 
(\cite[Theorem 6.3]{levins}) for homologically finite complexes of finite AB-dimension.

(3) If $R$ is an AB ring, then $\operatorname{AB-dim}_R(M)< \infty$ for all complexes $M \in D_b^f(R)$. Indeed, it follows by \cite[Theorem 19.5.8]{LarsBook} and \cite[Proposition 3.6]{levins}.
\end{remark}

The following result was proved by Araya in \cite[Theorem 1.2(2)]{araya} for finitely generated modules over local rings.
\begin{proposition}\label{prop:ineq2}
Let $M \in D_b^f(R)$ be a complex. The following results hold:
\begin{enumerate}
    \item There are inequalities
\[\operatorname{G-dim}_R(M) \leq \operatorname{AB-dim}_R(M) \leq \operatorname{CI-dim}_R(M);\]
and when one of these dimensions is finite it is equal to those on its left. In particular, $\sup (M) \leq \operatorname{AB-dim}_R(M)$.
    \item If $R$ is local and $\operatorname{AB-dim}_R(M) < \infty$, then $\operatorname{AB-dim}_R(M)=\depth R - \depth M$. 
\end{enumerate}
\end{proposition}
\begin{proof}

(1) We may assume that $M$ is non-acyclic. If $\operatorname{AB-dim}_R(M) < \infty$, then by definition one has $\gdim_R(M) < \infty$, and the string of equalities $\operatorname{AB-dim}_R(M)= \gdim_R(M)=B_M^R$ follows from \cite[Theorem 4.2]{levins}. If we have $\operatorname{CI-dim}_R(M)< \infty$, then $\gdim_R(M)=\operatorname{CI-dim}_R(M)$ by \cite[Proposition 3.3]{sather}. Thus, it remains only to prove that $B_M^R < \infty$ provided that $M \in D_b^f(R)$ is a  complex of finite complete intersection dimension. Let \(P \stackrel{\simeq}{\longrightarrow} M\) be a degreewise finitely generated semi-projective resolution such that $P_i=0$ for all $i< \inf(M)$. Set $s=\sup(M)$ and let $C=C_s(P) \neq 0$. It follows from \cite[Proposition 3.7(ii)]{sather} that $\operatorname{CI-dim}_R(C)< \infty$. Therefore, $b_C^R< \infty$ by \cite[Theorem 4.7]{avramov}. Since $\gdim_R(M)< \infty$, applying \cite[Proposition 3.6]{levins} yields $B_M^R< \infty$, as desired.

(2) From item (1), we have the equalities $\operatorname{AB-dim}_R(M)= \gdim_R(M)= \depth R - \depth M$.
\end{proof}

\begin{proposition}\label{prop:coker}
Let $M \in D_b^f(R)$ be a complex.  Let $n= \operatorname{AB-dim}_R(M)<\infty$ and $P \stackrel{\simeq}{\longrightarrow} M$ be a degreewise finitely generated semi-projective resolution of $M$ such that $P_i =0$ for all $i< \inf(M)$. If $C_n(P) \neq 0$, then $\operatorname{AB-dim}_R(C_n(P))=0$.
\end{proposition}
\begin{proof}
Since $n \geq \sup (M)$, then $b_{C_n(P)}^R< \infty$ by \cite[Proposition 3.6]{levins}. Moreover, we have that $\gdim_R(C_n(P))=0$ (see e.g., \cite[Theorem 19.4.1]{LarsBook}). Therefore,  $\operatorname{AB-dim}_R(C_n(P))< \infty$, and it follows from Proposition \ref{prop:ineq2}(1) that $\operatorname{AB-dim}_R(C_n(P))= \gdim_R(C_n(P))=0$. 
\end{proof}
The following proposition describes the behavior of AB-dimension over exact triangles and will play a crucial role in the proof of our main result.
\begin{proposition}\label{triangleLemma}
Let $M_1 \to M_2 \to M_3 \leadsto $ be an exact triangle of complexes in $D_b^f(R)$, and let $\{i,j,k\}=\{1,2,3\}$. If $\operatorname{AB-dim}_R(M_i)< \infty$  and $\operatorname{pd}_R(M_j)<\infty$, then $\operatorname{AB-dim}_R(M_k)< \infty$.
\end{proposition}
\begin{proof}
It follows from \cite[Theorem 9.1.16]{LarsBook} that $\gdim_R(M_k) < \infty$. It remains to  prove that $B_{M_k}^R< \infty$. Let $N \in D_b^f(R)$ be a complex with $\operatorname{P}_R(M_k,N) < \infty$. Since $\operatorname{P}_R(M_k,N)$ is invariant under shifting $N$, we may assume that $\inf(N)=0$. Thus, $\operatorname{P}_R(M_k,N)= - \inf (\RHom_R(M_k,N))$. The given exact triangle induces the following exact triangle:
\[\RHom_R(M_3,N) \rightarrow \RHom_R(M_2,N) \rightarrow \RHom_R(M_1,N) \leadsto.\]

Considering its long exact sequence induced on homologies, note that $\operatorname{P}_R(M_i, N)<\infty$ as $\operatorname{pd}_R(M_j)< \infty$ and $\operatorname{P}_R(M_k, N)<\infty$, see \cite[Theorem 8.1.8]{LarsBook}. In addition,  
\[\operatorname{P}_R(M_k,N) \leq \sup \lbrace \operatorname{P}_R(M_i,N)+1,\operatorname{P}_R(M_j,N)+1 \rbrace \leq \sup \lbrace B_{M_i}^R+1, \pd_R(M_j)+1 \rbrace.\]
Hence $B_{M_k}^R \leq \sup \lbrace B_{M_i}^R+1, \pd_R(M_j)+1 \rbrace< \infty$, as required.
\end{proof}

\subsection{AB-dimension and vanishing of Tor}
In this subsection, we study the vanishing of Tor assuming finite AB-dimension. The results obtained here will be essential for establishing the main results of the paper.
\begin{proposition} \label{L_TorAuslanderBound} Let $R$ be a Noetherian ring with a dualizing complex $D$, and let $M$ and $N$ be complexes in $D_{b}^{f}(R)$ with $B_M^R< \infty$. If $\Tor_i^R(M,N)=0$ for all $i \gg 0$, then
\[\sup{(M\ldt_{R}N)} \leq B_{M}^{R} + \operatorname{id}_{R}(D) + \sup{(D)} + \sup{(N)}.\]
\end{proposition}
\begin{proof}  
As $M\ldt_{R}N$ is in $D_{b}^{f}(R)$ and since there is an isomorphism
\[\RHom_{R}(M\ldt_{R}N,D) \simeq \RHom_{R}(M,\RHom_{R}(N,D)),\]
$\RHom_{R}(M,\RHom_{R}(N,D))$ is in $D_{b}^{f}(R)$. This gives the second equality below
\begin{align*}
\sup{(M\ldt_{R}N)} &= \sup{\big(\RHom_{R}(\RHom_{R}(M\ldt_{R}N,D),D)\big)} \\
&= \sup{\big(\RHom_{R}(\RHom_{R}(M,\RHom_{R}(N,D)),D)\big)} \\
&\leq -\inf{\big(\RHom_{R}(M,\RHom_{R}(N,D))\big)} + \sup{(D)} \\
&\leq B_{M}^{R} - \inf{(\RHom_{R}(N,D))} + \sup{(D)} \\
&\leq B_{M}^{R} + \id_{R}(D) + \sup{(N)} + \sup{(D)},
\end{align*}
The first inequality is by \cite[Proposition 7.6.7]{LarsBook}, the second follows from the definition of $B_M^R$, and the third inequality is by \cite[Theorem 8.2.8]{LarsBook}.
\end{proof}

\begin{lemma} \label{L_ABdimZero_TorIndependent} Let $R$ be a Noetherian ring with a dualizing complex $D$, and let $M$ and $N$ be finitely generated $R$-modules with $\operatorname{AB-dim}_R (M)=0$. If $\operatorname{Tor}_{i}^{R}(M,N)=0$ for all $i\gg 0$, then $\operatorname{Tor}_{i}^{R}(M,N)=0$ for all $i>0$.
\end{lemma}

\begin{proof}
Since $\operatorname{AB-dim}_R( M)=0$, there exists a totally acyclic complex
\[ T:\cdots\rightarrow  T_{1}\xrightarrow{\partial_{1}} T_{0}\xrightarrow{\partial_{0}} T_{-1}\rightarrow\cdots\]
of finitely generated free $R$-modules with $M=C_0(T)$. For simplicity, denote $C_i=C_i(T)$. From the complex above we get an exact sequence
\(0\rightarrow M\rightarrow T_{-1}\rightarrow C_{-1}\rightarrow 0\),
and then we have $\operatorname{Tor}_{i}^{R}(M,N)\cong \operatorname{Tor}_{i+1}^{R}(C_{-1},N)$ for $i\geq 1$. By an induction argument, we get $\operatorname{Tor}_{i}^{R}(M,N)\cong \operatorname{Tor}_{i+j}^{R}(C_{-j},N)$ for $i\geq 1$ and $j\geq 0$. Let $e=\operatorname{id}_{R}(D) + \sup{(D)}$ and suppose $\operatorname{Tor}_{n}^{R}(M,N)\neq 0$ for some $n\geq 1$. Note that $e \geq 0$ (see e.g., \cite[8.2.8]{LarsBook}). 
 By construction, each $C_i$ is totally reflexive, and then by Proposition \ref{triangleLemma}, its $\operatorname{AB}$-dimension is zero. In particular, $B^R_{C_{-e}}=0$.
We have that
\[\operatorname{Tor}_{n+e}^{R}(C_{-e},N) \cong \operatorname{Tor}_{n}^{R}(M,N) \neq 0,\]
but this gives
\(\sup{(C_{-e}\ldt_{R}N)} \geq n + e > \operatorname{id}_{R}(D) + \sup{(D)}\), thus
contradicting the bound given in Proposition \ref{L_TorAuslanderBound}.
\end{proof}

The following proposition is a version of \cite[Lemma 6.2]{levins} that does not require the ring to be local, and its proof follows the same argument. While the statement in \cite{levins} is formulated without assuming that the ring is Cohen-Macaulay, its proof uses \cite[Proposition 7.5.5]{sanders}, which is established over Cohen-Macaulay local rings. For completeness, we give here the corrected proof over Noetherian rings admitting a dualizing complex.

\begin{proposition}\label{torvanishing} Let $R$ be a Noetherian ring with a dualizing complex, let $M$ be a finitely generated $R$-module with $\operatorname{AB-dim}_R(M)=0$, and let $N\in D_{b}^{f}(R)$ be a bounded complex. If $\Tor_i^R(M,N)=0$ for all $i \gg 0$, then $\widehat{\operatorname{Tor}}{}_i^R(M,N)=0$ for all $i
\in\mathbb{Z}$.
\begin{proof}
We can assume that $N_0 \neq 0$ and $N_i=0$ for all $i>0$. Let $T$ be a totally acyclic complex of finitely generated free $R$-modules with $M=C_0(T)$. Fix $n \in \mathbb{Z}$ and denote $C=C_{n-1}(T)$. Note that $\gdim_R(C)=0$, $b_{C}^R< \infty$ and $\operatorname{Tor}_{i}^R (C,N)=0$ for all $i \gg 0$.  Thus $\operatorname{AB-dim}_R(C)=0$ by \cite[Lemma 3.5 and Corollary 4.3]{levins}. Moreover 
\[\widehat{\Tor}_n^R(M,N) \cong \widehat{\Tor}_1^R(C,N) \cong \Tor_1^R(C,N),\]
where the second isomorphism follows from \cite[2.4.1]{LarsDave}. Let $F \stackrel{\simeq}{\longrightarrow} N$ be a degreewise finitely generated semi-projective resolution of $N$. Follows from \cite[Lemma 8.3.9]{LarsBook} that $\Tor_j^R(C,N) \cong \Tor_j^R(C,C_0(F))$ for all $j>0$. Thus $\Tor_j^R(C,C_0(F))=0$ for all $j \gg 0$. Therefore, from the above isomorphisms and Lemma \ref{L_ABdimZero_TorIndependent}, we have 
\[\widehat{\Tor}_n^R(M,N) \cong \Tor_1^R(C,N) \cong \Tor_1^R(C,C_0(F))=0.\qedhere\]
\end{proof}
\end{proposition}



\section{Generalized symmetry in the vanishing of Ext}\label{section:main}

In this section, we investigate a Generalized Symmetry in the vanishing of Ext over Noetherian rings with a dualizing complex. To prove the main result of this paper, we prepare auxiliary lemmas. 

The following lemma was first proved in the setting of maximal Cohen-Macaulay modules over Gorenstein local rings in \cite[Theorem 2.1]{JorgensenHuneke}, and later in the setting of Cohen-Macaulay modules over Cohen-Macaulay local rings with a canonical module in \cite[Theorem 5.2]{Ghosh-Puth}. Here, we extend these results to a more general framework using homologically finite complexes over a Noetherian ring with a dualizing complex.
\begin{lemma}\label{lemma1}Let $R$ be a Noetherian ring with a dualizing complex $D$, and let $M$ and $N$ be complexes in $D_b^f(R)$. The following are equivalent:
\begin{enumerate}
    \item $\RHom_R(M, \RHom_R(N,D))$ is homologically bounded.
    \item $\RHom_R(N, \RHom_R(M,D))$ is homologically bounded.
    \item $M \ldt_R N$ is homologically bounded.
\end{enumerate}
\end{lemma}
\begin{proof}
By Grothendieck Duality \cite[Theorem 18.2.3]{LarsBook} and \cite[E 10.1.7]{LarsBook}, we have the isomorphism $\RHom_R(M, \RHom_R(N,D)) \simeq \RHom_R(N, \RHom_R(M,D))$ in $D(R)$. This shows the equivalence between $(1)$ and  $(2)$. The implication  $(1) \Rightarrow (3)$ follows from the following isomorphisms in $D(R)$: 
 \[M \ldt_R N  \simeq M \ldt_R \RHom_R(\RHom_R(N,D), D) \simeq \RHom_R(\RHom_R(M, \RHom_R(N,D)), D),\]
where the second isomorphism follows from \cite[Corollary 12.3.26(b)]{LarsBook}.
 Also, one can check that $(3) \Rightarrow (1)$ using the isomorphism $\RHom_R(M, \RHom_R(N,D))\simeq \RHom_R(M \ldt_R N, D).$
\end{proof}

For a finitely generated $R$-module $M$, we denote $M^*=\Hom_R(M,R)$.

\begin{lemma}\label{lemma:avramov}
Let $M$ be a finitely generated $R$-module with $\gdim_R(M)=0$, and let $N$ be a bounded complex in $D_b^f(R)$. Then 
\[\widehat{\operatorname{Ext}}{}_R^{-i-1}(M,N) \cong \widehat{\operatorname{Tor}}{}_i^R(M^*,N)\] 
for all $i \in \mathbb{Z}.$ 
\end{lemma}
\begin{proof}
Let $L \stackrel{\simeq}{\longrightarrow} N$ be a degreewise finitely generated semi-projective resolution of $N$ and set $n \in \mathbb{Z}$ such that $N_j =0$ for all $j \geq n$. For each $i \in \mathbb{Z}$, we have: 
\[ \widehat{\operatorname{Ext}}{}_R^{-i-1}(M,N) \cong \widehat{\operatorname{Ext}}{}_R^{-i-1+n}(M,C_n(L)) \cong \widehat{\operatorname{Tor}}{}_{i-n}^R(M^*,C_n(L))  \cong  \widehat{\operatorname{Tor}}{}_{i}^R(M^*,N),  \]
where the first isomorphism is given by \cite[Lemma 4.1]{levins}, the second is \cite[4.4.7]{avramov}, and the third isomorphism in given in \cite[Lemma 2.10]{LarsDave}.
\end{proof}
Before proving our main theorem, we first establish the following case, which will simplify the proof of the main result and improve readability.
\begin{proposition}\label{m1} Let $R$ be a Noetherian ring with a dualizing complex $D$, let $M$ be a finitely generated $R$-module with 
$\operatorname{AB-dim}_R (M)=\operatorname{AB-dim}_R (M^*)=0$, and let $N$ be a complex in $D_b^f(R)$. Then 
\[\Ext_R^i(M,N)=0 \text{ for all }i\gg 0 \Longleftrightarrow \Ext_R^i(N,M \ldt_R D)=0 \text{ for all } i  \gg 0 \]
\end{proposition}
\begin{proof}
We may assume that $N$ is bounded (up to a quasi-isomorphism, which does not affect the eventual vanishing of Ext). Note that $\gdim_R(M)=\gdim_R(M^*)=0$. By Lemma \ref{lemma:avramov}, we have 
\[\widehat{\operatorname{Ext}}{}_R^{-i-1}(M,N) \cong \widehat{\operatorname{Tor}}{}_i^R(M^*,N)\] 
for all $i \in \mathbb{Z}$. Using this isomorphism and \cite[Lemma 3.5 and Lemma 4.4]{levins}, we see that: 
$\operatorname{Ext}_R^{i} (M,N) =0$ for all $i \gg 0$ if and only if $\widehat{\operatorname{Ext}}{}_R^{i}(M,N)=0$ for all $i \in \mathbb{Z}$ if and only if $\widehat{\operatorname{Tor}}{}_i^R(M^*,N)=0$ for all $i \in \mathbb{Z}$. Now, since $\operatorname{AB-dim}_R(M^*)=0$, using Proposition \ref{torvanishing} we see that $\widehat{\operatorname{Tor}}{}_i^R(M^*,N)=0$ for all $i \in \mathbb{Z}$ if and only if $\operatorname{Tor}_{i} ^R(M^*,N)=0$ for all $i \gg 0$.

Moreover, using the equivalences in Lemma \ref{lemma1} we see that $\operatorname{Tor}_{i} ^R(M^*,N)=0$ for all $i \gg 0$ if and only if $\Ext_R^i(N, \RHom_R(M^*,D))=0$ for all $i \gg 0$. Also, note that: \[\RHom_R(M^*, D)\simeq \RHom_R(\RHom_R(M,R), D) \simeq M \ldt_R \RHom_R(R, D)\simeq M \ldt_R D,\]
where the second isomorphism follows from \cite[Corollary 12.3.26]{LarsBook}. Therefore, we have that: $\Ext_R^i(N, \RHom_R(M^*,D))=0$ for all $i\gg 0$ if and only  $\Ext_R^i(N, M\ldt_R D)=0$ for all $i\gg 0$. Combining the equivalences established above, the result follows. 
\end{proof}
We now state the main theorem of this paper, which we call the Generalized Symmetry of Ext.
\begin{theorem}\label{main} Let $R$ be a Noetherian ring with a dualizing complex $D$, and let $M$ and $N$ be  complexes in $D_b^f(R)$ with 
$\operatorname{AB-dim}_R(M)< \infty$ and $\operatorname{AB-dim}_R( \RHom_R(M,R))< \infty$. Then 
\[\Ext_R^i(M,N)=0 \text{ for all }i\gg 0 \Longleftrightarrow \Ext_R^i(N,M \ldt_R D)=0 \text{ for all } i  \gg 0. \]
\end{theorem}
\begin{proof}
Set $n=\operatorname{AB-dim}_R (M)$ and let $P \stackrel{\simeq}{\longrightarrow} M$ be a degreewise finitely generated semi-projective resolution such that $P_i=0$ for all $i<\inf M$. We may assume that $C_n(P) \neq 0$, because otherwise $\operatorname{pd}_R(M)<\infty$, and the conclusion follows. Then $0=\operatorname{AB-dim}_R(C_n(P))=\gdim_R(C_n(P))$ by Proposition \ref{prop:coker}. 

    Now considering the short exact sequence of complexes $0 \rightarrow P_{\leq n-1} \rightarrow P \rightarrow P_{\geq n} \rightarrow 0$ and  using that $P_{\geq n} \simeq \Sigma^n C_n(P)$,
    we have an exact triangle 
\begin{equation}\label{impor}
        P_{\leq n-1}  \to P \to  \Sigma^n C_n(P) \leadsto.
    \end{equation} 
This induces the exact triangle
\[\RHom_R(\Sigma^n C_n(P), R) \to \RHom_R(M,R) \to \RHom_R( P_{\leq n-1}, R)  \leadsto \]
 Since $\operatorname{pd}_R (\RHom_R(P_{\leq n-1}, R))$ and $\operatorname{AB-dim}_R (\RHom_R(M,R))$ are both finite, it follows from Proposition \ref{triangleLemma} that $\RHom_R(C_n(P), R)\simeq (C_n(P))^*$ has finite AB-dimension. Moreover, \linebreak $\operatorname{AB-dim}_R ((C_n(P))^*)=0$ since $C_n(P)$ is totally reflexive. 

On other hand, using again \eqref{impor}, we have the following two exact triangles
\begin{align}
    \RHom_R(\Sigma^n C_n(P), N)& \to   \RHom_R(M,N) \to \RHom_R(P_{\leq n-1}, N) \leadsto \label{4.2}\\
    \RHom_R(N, P_{\leq {n-1}} \ldt_R D)& \to  \RHom_R(N, M \ldt_R D) \to  \RHom_R(N, \Sigma^n C_n(P) \ldt_R D)\leadsto \label{4.3}
\end{align}
Since $\operatorname{pd}_R (P_{\leq n-1})<\infty$, note that $\RHom_R(P_{\leq n-1}, N)$ and  $\RHom_R(N, P_{\leq {n-1}} \ldt_R D)$ are homologically bounded, so we have that:     \begin{itemize}
        \item From (\ref{4.2}): $\Ext_R^i(M,N)=0$ for all $i\gg 0$ if and only if $\Ext_R^i(C_n(P), N)=0$ for all $i\gg 0$.
        \item From (\ref{4.3}):
        $\Ext_R^i(N,M \ldt_R D)=0$ for all $i\gg 0$ if and only if $\Ext_R^i( N, C_n(P) \ldt_R D)=0$ for all $i\gg 0$.
\end{itemize}
However, by Proposition \ref{m1}, we have that $\Ext_R^i(C_n(P), N)=0$ for all $i \gg 0$ is equivalent to $\Ext_R^i(N, C_n(P)\ldt_R D)=0$ for all $i \gg 0$. Therefore, $\Ext_R^i(M,N)=0$ for all $i \gg0$ if and only if $\Ext_R^i(N,M \ldt_R D) =0 $ for all $i \gg 0$, as required. 
\end{proof}
\begin{remark}\label{rem}

It remains unknown for us if, for a complex $M \in D_b^f(R)$, the assumption that $\operatorname{AB-dim}_R(M)< \infty$ in our main result implies that $\operatorname{AB-dim}_R(\RHom_R(M,R))< \infty$.  In Subsection \ref{subsection:5.1}, we suggest a potential approach to show this and then remove the assumption that $\operatorname{AB-dim}_R(\RHom_R(M,R))$ is finite. Even if that is true, we remark that $M$ having finite AB-dimension plays a crucial role to have Generalized Symmetry in the vanishing of Ext as in Theorem \ref{main}. That is, we observe that the equivalence in Theorem \ref{main} may fail if $M$ does not have finite Gorenstein dimension and finite Auslander bound. 
\begin{enumerate}
    \item Jorgensen and Şega \cite{JO-SE} constructed an intriguing example showing that symmetry in the vanishing of Ext may fail over a Gorenstein local ring. Their example is given in \cite[Corollary 4.2]{JO-SE}. Moreover, in \cite[(4.14)]{JO-SE}, they show that the module $M$ appearing in their construction has infinite Auslander bound.
    \item Let $(R,\mathfrak{m},k)$ be a local ring with a dualizing complex $D$. Note that $b_k^R< \infty$. Considering $M=k$ and $N=R$, one can see that $\Ext_R^i(R,k \ldt_R D) =0$ for all $i \gg 0$ (see \cite[Theorem 8.1.8]{LarsBook}), but $\operatorname{Ext}_R^i(k,R)=0$ for all $i \gg0$ does not happen unless $R$ is Gorenstein, that is, $\gdim_R(k)< \infty$. 
\end{enumerate}
\end{remark}
\section{Corollaries of the main result}\label{section:corollaries}
In this section, we derive corollaries of our main result that extend several symmetry results for the vanishing of Ext from finitely generated modules over Gorenstein local rings to homologically finite complexes over Noetherian rings with a dualizing complex. Compared to earlier results in the literature, our results require weaker assumptions on the ring and apply in the more general setting of homologically finite complexes.

The following corollary provides an extension of the foundational result of Huneke and Jorgensen for AB rings \cite[Theorem 4.1]{JorgensenHuneke}. It should also be compared with \cite[Theorem B.2]{CH}.

\begin{corollary}\label{corollary}
Let $R$ be a Noetherian ring with a dualizing complex $D$, and assume that $R$ satisfies \textnormal{(AC)}. Let $M$ and $N$ be complexes in $D_b^f(R)$ with $\gdim_R(M)< \infty$. The following conditions are equivalent:
\begin{enumerate}
    \item $\Ext_R^i(M,N)=0$ for all $i \gg 0$.
    \item $\Ext_R^i(N,M \ldt_R D)=0$ for all $i \gg 0$.
    \item $\Tor_i^R(M,\RHom_R(N,D))=0$ for all $i \gg 0$. 
\end{enumerate}
\end{corollary}
\begin{proof}
One can see from \cite[Theorem 19.4.18]{LarsBook} that $\gdim_R(\RHom_R(M,R))< \infty$. Therefore, since $R$ satisfies (AC), we have $B_M^R< \infty$ and $B_{\RHom_R(M,R)}^R < \infty$ by \cite[Proposition 3.6]{levins}. That is, $\operatorname{AB-dim}_R(M)$ and $\operatorname{AB-dim}_R(\RHom_R(M,R))$ are both finite. The equivalence (1) $\Leftrightarrow$ (2) is then a consequence of Theorem \ref{main} and (1) $\Leftrightarrow$ (3) follows by Lemma \ref{lemma1}.
\end{proof}
\begin{remark} 
\begin{enumerate}
   \item  Note that Remark \ref{rem}(2) shows the assumption $\gdim_R(M) < \infty$ in Corollary \ref{corollary} cannot be removed.
    \item Interesting examples of non-Gorenstein rings satisfying (AC) can be found in the literature; see, for instance, \cite[Example 3.2]{LarsHolmAC} and \cite[Theorem A]{TakahashiUAC}. Moreover, Cohen-Macaulay local rings of finite Cohen-Macaulay type, Golod rings, and trivial extensions of a local ring by its residue field satisfy (AC); see \cite[Theorem 1.2]{LarsHolmAC}, \cite[Proposition 1.4]{JO-LI}, and \cite[Proposition 3.7]{IGOR}, respectively.
\end{enumerate}
\end{remark}


A non-zero $R$-module $M$ is said to have finite \textit{quasi-projective dimension} provided that there exists a bounded complex $P$ of projective $R$-modules, not acyclic, whose homology modules are either zero or finite direct sums of copies of $M$. The quasi-projective dimension of $M$ is denoted by $\operatorname{qpd}_R(M)$. The quasi-projective dimension of the zero module is set to be $-\infty$. For more on quasi-projective dimension, see \cite{GJT}.

\begin{lemma}\label{victor:lemma}
Let $M$ be a finitely generated $R$-module with $\operatorname{qpd}_R(M) < \infty$ and $\gdim_R(M)<\infty$. Then $\operatorname{AB-dim}_R(\RHom_R(M,R))<\infty$.
\end{lemma}
\begin{proof}
First, we prove the case in which $\gdim_R(M)=0$. In this case $\RHom_R(M,R) \simeq M^*$, $\gdim_R(M^*)=0$ and $\operatorname{qpd}_R(M^*)< \infty$; see \cite[Proposition 6.14]{GJT}. Then, using \cite[Corollary 6.4]{GJT}, one can see that $b_{M^*}^R< \infty$ and then $\operatorname{AB-dim}_R(M^*)< \infty$, as required.

To prove the general case, let $P \stackrel{\simeq}{\longrightarrow} M$ be a projective resolution of $M$ and $n=\gdim_R(M)$. Consider the short exact sequence of complexes $0 \rightarrow P_{\leq n-1} \rightarrow P \rightarrow P_{\geq n} \rightarrow0$ and using that $P_{\geq n} \simeq \Sigma^n C_n(P)$, we have an exact triangle
\begin{align}\label{trangle3}
\RHom_R(\Sigma^n C_n(P), R) \to \RHom_R(M,R) \to \RHom_R( P_{\leq n-1}, R)  \leadsto. 
\end{align}
Clearly, $\pd_R(\RHom_R(P_{\leq n-1},R))< \infty$. By \cite[Proposition 3.3(4)]{GJT}, $\operatorname{qpd}_R(C_n(P))< \infty$. Also, $\gdim_R(C_n(P))=0$ and $\RHom_R(C_ n(P),R)$ have  finite AB-dimension, from the case proved previously. Thus, applying Proposition \ref{triangleLemma} to the exact triangle (\ref{trangle3}) we see that $\RHom_R(M,R)$ has finite AB-dimension.
\end{proof}

The following corollary should be compared with \cite[Theorem 6.16]{GJT}. Compared with this interesting result of Gheibi, Jorgensen, and Takahashi, our result considers non-Gorenstein Noetherian rings with a dualizing complex and extends the conclusion to the setting where \(N\) is a homologically finite complex.
\begin{corollary}\label{cor:qpd}
Let $R$ be a Noetherian ring with a dualizing complex $D$, and let $M$ be a finitely generated $R$-module with $\operatorname{qpd}_R(M)< \infty$ and $\gdim_R(M)< \infty$. Let $N$ be a complex in $D_b^f(R)$.  Then 
\[\Ext_R^i(M,N)=0 \text{ for all }i\gg 0 \Longleftrightarrow \Ext_R^i(N,M \ldt_R D)=0 \text{ for all } i  \gg 0. \] 
\end{corollary}
\begin{proof}
Since $\operatorname{qpd}_R(M)< \infty$, then $b_M^R< \infty$ by \cite[Corollary 6.4]{GJT}, and so $\operatorname{AB-dim}_R(M)< \infty$. Moreover, by Lemma \ref{victor:lemma}, $\operatorname{AB-dim}_R(\RHom_R(M,R))< \infty$. The desired equivalence follows then from Theorem \ref{main}.
\end{proof}

\subsection{A general approach and further considerations}\label{subsection:5.1}
In order to suggest a potential approach to drop the assumption $\operatorname{AB-dim}_R(\RHom_R(M,R))$ in Theorem \ref{main}, we pose the following question:
\begin{question}\label{question5.5}
Let $M$ be a complex in $D_b^f(R)$. If $\operatorname{AB-dim}_R(M)<\infty$, does it follow that $\operatorname{AB-dim}_R(\RHom_R(M,R))< \infty$?
\end{question}

Note that Lemma \ref{victor:lemma} gives a particular answer to Question \ref{question5.5}. To provide an application to complexes of finite complete intersection dimension and to formulate Question \ref{question5.5} equivalently in terms of finitely generated modules, we establish the following lemma.

\begin{lemma}\label{lemma:equivalences}
Let $\Hdim$ be a homological invariant defined on $D_b^f(R)$ with the following properties:  
\begin{enumerate}
    \item $\gdim_R(M) \leq \Hdim_R(M) \leq \pd_R(M)$ for all $M \in D_b^f(R)$, and when one of these invariants is finite it is equal to those on its left. 
    \item Let $M_1 \rightarrow M_2 \rightarrow M_3 \leadsto$ be an exact triangle in $D_b^f(R)$. Let $\lbrace i,j,k \rbrace= \lbrace 1,2,3 \rbrace$. If $M_i$ has finite $\Hdim$ and $M_j$ has finite projective dimension, then $M_k$ has finite $\Hdim$. 
    \item Finiteness of $\Hdim$ is preserved under shifting in $D_b^f(R)$.
\end{enumerate}
Then the following conditions are equivalent:
\begin{enumerate}
    \item[(a)] For each finitely generated $R$-module $M$, one has
    \[\Hdim_R(M) =0 \Longrightarrow \Hdim_R(M^*) =0.\] 
    \item[(b)] For each complex $M \in D_b^f(R)$, one has
    \[\Hdim_R(M) < \infty \Longrightarrow \Hdim_R(\RHom_R(M,R))< \infty.\]
\end{enumerate}
\end{lemma}
\begin{proof}
The implication (b) $\Rightarrow$ (a) is trivial. We show that (a) $\Rightarrow$ (b). Let $M \in D_b^f(R)$ be a complex. Let $P \stackrel{\simeq}{\longrightarrow}M$ be a degreewise finitely generated semi-projective resolution with $P_i=0$ for all $i < \inf M$. Set $n=\Hdim_R(M)$. If $C_n(P)=0$, then $\pd_R(M) < \infty$ and the result follows. So, assume that $C_n(P) \neq 0$. Consider the exact sequence $0 \rightarrow P_{\leq n-1} \rightarrow P \rightarrow P_{\geq n} \rightarrow0$. Since $P_{\geq n} \simeq \Sigma^n C_n(P)$, one can see from (1-3) that  $\Hdim_R(C_n(P))=\gdim_R(C_n(P))=0$. Moreover, we have the following exact triangle in $D_b^f(R)$:
\begin{align}\label{triangleN}
\RHom_R(\Sigma^n C_n(P),R) \rightarrow \RHom_R(M,R) \rightarrow \RHom_R(P_{\leq n-1},R) \leadsto.
\end{align}
Note that, using (a), $\RHom_R(C_n(P),R) \simeq (C_n(P))^*$ has finite $\Hdim$ and so also the complex $\RHom_R(\Sigma^nC_n(P),R) \simeq \Sigma^{-n} \RHom_R(C_n(P),R)$ from  (3). Finally, since we have that $\operatorname{pd}_R(\RHom_R(P_{\leq n-1},R))< \infty$, the conclusion follows using the exact triangle (\ref{triangleN}) and (2).
\end{proof}

\begin{remark}\label{ci-dim}
The complete intersection dimension (CI-dim) is a homological invariant defined on $D_b^f(R)$ satisfying the conditions (1)-(3) and (a) (and hence (b)) of Lemma \ref{lemma:equivalences}. 
Indeed, (1) is established in \cite[Proposition 3.3]{sather}, (2) follows from a similar argument that \cite[Lemma 3.6]{sather}, (3) is straightforward to verify. Moreover, to show that CI-dim satisfies (a), we can reduce to the local case and then apply \cite[Lemma 3.5]{be-jo}. Therefore, if $M \in D_b^f(R)$ is a complex of finite CI-dimension, then $\RHom_R(M,R)$ has finite CI-dimension.
\end{remark}

\begin{setup}\label{setup}
Consider $\Hdim$ a homological invariant defined on $D_b^f(R)$ satisfying conditions (1)-(3) and (a) (and hence (b)) of Lemma \ref{lemma:equivalences}, and assume that $\operatorname{AB-dim}_R(M) \leq \Hdim_R(M)$ for every complex $M \in D_b^f(R)$. 
\end{setup}
The following corollary extends the symmetry in the vanishing of Ext over Gorenstein local rings proved by Jørgensen in \cite[Theorem 4.1]{PE-JO} for finitely generated modules of finite complete intersection dimension. Our result not only strengthens Jørgensen's result by relaxing the assumptions on the ring, but also extends the conclusion to homologically finite complexes assuming that only one of them has finite complete intersection dimension.
\begin{corollary}\label{cor:hdim}
Let $R$ be a Noetherian ring with a dualizing complex $D$. Consider $\Hdim$ to be a homological invariant defined in $D_b^f(R)$ as in Setup \ref{setup}. Let $M$ and $N$ be complexes in $D_b^f(R)$ with $\Hdim_R(M)< \infty$. Then 
\[\Ext_R^i(M,N)=0 \text{ for all }i\gg 0 \Longleftrightarrow \Ext_R^i(N,M \ldt_R D)=0 \text{ for all } i  \gg 0. \]
In particular, if $\Hdim=\operatorname{CI-dim}$, the above equivalence holds.
\end{corollary}
\begin{proof}
We have $\operatorname{AB-dim}_R(M) \leq \Hdim_R(M)< \infty$ by assumption.  By Lemma \ref{lemma:equivalences} we have $\Hdim_R(\RHom_R(M,R))< \infty$ and so $\operatorname{AB-dim}_R(\RHom_R(M,R))< \infty$. The conclusion then follows by Theorem \ref{main}. The second part follows then using Remark \ref{ci-dim} and Proposition \ref{prop:ineq2}(1).
\end{proof}
Let us suggest a potential approach to drop the assumption $\operatorname{AB-dim}_R(\RHom_R(M,R))< \infty$ in Theorem \ref{main}. In Section \ref{section:AB-dim}, we showed that the AB-dimension satisfies conditions (1-2) in Lemma \ref{lemma:equivalences}. Also, (3) is straightforward to verify. 
By Lemma \ref{lemma:equivalences}, Question \ref{question5.5} is then equivalent to asking whether, for a finitely generated $R$-module $M$, the condition $\operatorname{AB-dim}_R(M)=0$ implies $\operatorname{AB-dim}_R(M^*)=0$. Thus, to address Question \ref{question5.5} and possibly drop the assumption $\operatorname{AB-dim}_R(\RHom_R(M,R))< \infty$ in Theorem \ref{main}, we may instead ask the following equivalent question in terms of finitely generated modules:

\begin{question}
Let $M$ be a finitely generated $R$-module with $\gdim_R(M)=0$. If $b_M^R < \infty$, does it follow that $b_{M^*}^R < \infty$?
\end{question}

\subsection{Symmetry in the vanishing of Ext over Noetherian rings} In this subsection, we establish from our main theorem several cases in which symmetry in the vanishing of Ext holds over a Noetherian ring. At the end of this subsection, we give an example satisfying the conditions considered in our results.

\begin{lemma}\label{LarsLemma}
Let $R$ be a Noetherian ring  with a dualizing complex $D$, and let $M$ be a complex in  $D^f(R)$. If $M \ldt_R D$ is homologically bounded to the right, then $M$ is homologically bounded to the right. 
\end{lemma}
\begin{proof}  By \cite[Corollary 14.1.13]{LarsBook}, the following equality holds:
$$\inf \{ \inf (M_\mathfrak{p} \ldt_{R_\mathfrak{p}} D_\mathfrak{p})  \mid  \mathfrak{p} \in \Spec(R) \} = \inf (M \ldt_R D).$$
Moreover, by \cite[Proposition 16.2.5(a) and Theorem 16.2.9]{LarsBook}, for each $\mathfrak{p} \in \operatorname{Supp} (M)$, we have that 
$$\inf (M_\mathfrak{p} \ldt_{R_\mathfrak{p}} D_\mathfrak{p})=\inf (M_\mathfrak{p})+\inf (D_\mathfrak{p}), $$
and hence $$\inf (M_\mathfrak{p}) \geq \inf (M \ldt_R D)-\inf (D_\mathfrak{p}) \geq \inf (M \ldt_R D)-\sup (D_\mathfrak{p}) \geq \inf (M \ldt_R D)-\sup(D).$$ 
Therefore, $\inf (M) \geq \inf (M\ldt_R D)-\sup (D)$ by \cite[Corollary 14.1.13]{LarsBook}. This show that $M$ is homologically bounded to the right. 
\end{proof}


The next corollaries provide conditions for symmetry in the vanishing of Ext without assuming that the base ring is Gorenstein.
\begin{corollary}\label{corollary:gpd}
Let $R$ be a Noetherian ring with a dualizing complex $D$, and let $M$ and $N$ be complexes in $D_b^f(R)$ with $\operatorname{Gpd}_R(\RHom_R(N,M))< \infty$. If $\operatorname{AB-dim}_R(M)<\infty$ and $\operatorname{AB-dim}_R(\RHom_R(M,R))< \infty$, then
\[\Ext_R^i(M,N)=0 \text{ for all }i\gg 0 \Longleftrightarrow \Ext_R^i(N,M )=0 \text{ for all } i  \gg 0.\]
\end{corollary}
\begin{proof}
The condition $\Ext_R^i(M,N)=0$ for all $i \gg 0$, that is $\RHom_R(M,N) \in D_b^f(R)$, is equivalent to $\RHom_R(N,M \ldt_R D) \in D_b^f(R)$ by Theorem \ref{main}. By \cite[Theorem 3.4]{MMDN},  
\[\RHom_R(N,M \ldt_R D) \simeq \RHom_R(N,M) \ldt_R D,\]
and then $\Ext_R^i(M,N)=0$ for all $i \gg0$ is equivalent to $\RHom_R(N,M) \ldt_R D \in D_b^f(R)$. By Lemma \ref{LarsLemma} and \cite[Theorem 19.1.12]{LarsBook}, this is equivalent to $\RHom_R(N,M) \in D_b^f(R)$, that is, to $\Ext_R^i(N,M)=0$ for all $i \gg 0$.
\end{proof}
\begin{corollary}\label{cor:symmetrygpd}
    Let $R$ be a Noetherian ring with a dualizing complex, and let $M$ and $N$ be complexes in $D_b^f(R)$ with $\operatorname{Gpd}_R(\RHom_R(N,M))< \infty$. Then
\[\Ext_R^i(M,N)=0 \text{ for all }i\gg 0 \Longleftrightarrow \Ext_R^i(N,M )=0 \text{ for all } i  \gg 0.\]
holds under any of the following hypotheses
\begin{enumerate}
    \item $R$ satisfies \textnormal{(AC)} and $\gdim_R(M)< \infty$.
    \item $\operatorname{CI-dim}_R(M)< \infty$.
    \item $M$ is a finitely generated $R$-module, $\operatorname{qpd}_R(M)< \infty$ and $\gdim_R(M)< \infty$.
\end{enumerate}
\end{corollary}
\begin{proof}
    (1) The proof follows the same lines as (1) $\Leftrightarrow$ (2) in Corollary \ref{corollary}, using Corollary \ref{corollary:gpd} in place of Theorem \ref{main}.

    (2) It follows by Corollary \ref{corollary:gpd}, Remark \ref{ci-dim} and Proposition \ref{prop:ineq2}(1).

    (3) Note that $b_M^R< \infty$ (\cite[Corollary 6.4]{GJT}) and then $\operatorname{AB-dim}_R(M)< \infty$. The result then follows using Lemma \ref{victor:lemma} and then Corollary \ref{corollary:gpd}.
\end{proof}
Due to the condition $\operatorname{Gpd}_R(\RHom_R(N,M))< \infty$ in Corollaries \ref{corollary:gpd} and \ref{cor:symmetrygpd}, we ask when this might happen. The next remark shows a particular case. The argument below was shared with us by Kaito Kimura.
\begin{remark}\label{remark:GpdRHom} 

Let $R$ be a Noetherian ring with finite Krull dimension, and let $M$ and $N$ be complexes in $D_b^f(R)$. If $\Ext_R^i(N,M)$ has finite Gorenstein dimension for all $i \in \mathbb{Z}$, then $\operatorname{Gpd}_R(\RHom_R(N,M))$ is finite. Up to quasi-isomorphism and shifting, we may assume that $\RHom_R(N,M)_n = 0$ for $n>0$. Since $\Ext_R^i(N,M)$ has finite Gorenstein dimension for all $i \geq 0$, $\RHom_R(N,M)_{\supseteq n}$ has finite Gorenstein dimension for all $n \leq 0$ (see \cite[Lemma 2.1(2)]{HO-VH-VD}). Then, \cite[Proposition 9.4.1]{LarsBook} gives the following inequality for $n \leq 0$
\begin{align}\label{g-dimineq}
\gdim_R(\RHom_R(N,M)_{\supseteq n}) \leq  \operatorname{FPD}(R) + \sup \RHom_R(N,M)_{\supseteq n} \leq \operatorname{FPD}(R), 
\end{align}
where $\operatorname{FPD}(R)$ denotes the \textit{finitistic projective dimension} of $R$ (\cite[Definition 8.5.16]{LarsBook}). Since $R$ has finite Krull dimension, then $\operatorname{FPD}(R)<\infty$; see \cite[Proposition 17.4.27]{LarsBook}.

By \cite[Example 3.3.34 and Proposition 3.3.37]{LarsBook}, there is a short exact sequence
\[0\rightarrow \coprod_{n\leq 0} \RHom_R(N,M)_{\supseteq n} \rightarrow \coprod_{ n\leq 0} \RHom_R(N,M)_{\supseteq n} \rightarrow \RHom_R(N,M)\rightarrow 0,\]
and by \cite[Proposition 9.1.23]{LarsBook} and  the inequality in (\ref{g-dimineq})
\[\operatorname{Gpd}_R \left( \coprod_{n \leq 0} \RHom_R(N,M)_{\supseteq n} \right) = \sup_{n\leq 0} \lbrace \operatorname{Gpd}_R(\RHom_R(N,M)_{\supseteq n}) \rbrace\leq \operatorname{FPD}(R).\]
Therefore $\RHom_R(N,M)$ has finite Gorenstein projective dimension using the short exact sequence above and the two-out-of-three property for Gorenstein projective dimension.
\end{remark}

The next example illustrates the situation in Remark \ref{remark:GpdRHom} and Corollary \ref{cor:symmetrygpd}(2,3) over a non-Gorenstein local ring. We note that the results in the literature about the symmetry of the vanishing of Ext do not apply in this context.

\begin{example}\label{ex:nongorensteinSym}
Let $k$ be a field. Consider the following ring and the following modules
\begin{align*}
    R= \frac{k[[r,s,t,x]]}{(r^2,rs,s^2,t^2,x^2)}, \quad M= \frac{R}{(x)} \quad \text{and} \quad N=\frac{R}{(t)}.
\end{align*}
$R$ is a non-Gorenstein Artinian local ring (and hence Cohen-Macaulay). We claim that: 
\begin{enumerate}
    \item $\operatorname{CI-dim}_R(M) = \operatorname{CI-dim}_R(N)=0$.
    \item $\operatorname{Ext}_R^i(M,N)=\Ext_R^i(N,M)=0$ for all $i>0$ and $\pd_R(M)=\pd_R(N)= \infty$.
    \item $\operatorname{CI-dim}_R(\Hom_R(N,M))=0$ and hence $\RHom_R(N,M) \simeq \Hom_R(N,M)$ has $\operatorname{G}$-dimension zero.

    \item $\operatorname{qpd}_R(M)=\operatorname{qpd}_R(N)=0$.
\end{enumerate}
\end{example}

\begin{proof}

(1) Considering $Q= \frac{k[[r,s,t,x]]}{(r^2,rs,s^2,t^2)}$, we may set a trivial quasi-deformation $R \rightarrow  R \leftarrow Q$. Note that $\pd_Q(R)=1$ and $M=R/(x) \cong Q/(x)$. So, $\pd_Q(M)-\pd_Q(R)=1-1=0$ and $\operatorname{CI-dim}_R(M)=0$. A similar argument shows that $\operatorname{CI-dim}_R(N)=0$.

(2) Consider the minimal free resolution: 
\begin{align*}
    F= \cdots \rightarrow R \xrightarrow{x} R \xrightarrow{x} R \xrightarrow{x} R \rightarrow 0
\end{align*}
of $M$ over $R$. So, $\pd_R(M)= \infty$. The complex $\Hom_R(F,N)$ is 
\begin{align*}
    0 \rightarrow N \xrightarrow{x} N \xrightarrow{x} N\xrightarrow{x} N \rightarrow \cdots.
\end{align*}
Therefore, computing homologies we see that $\Ext_R^i(M,N)=xN/xN=0$ for all $i>0$. Again, a similar argument shows that $\Ext_R^i(N,M)=0$ for all $i>0$ and $\pd_R(N)= \infty$.

(3) One can see that $\Hom_R(N,M)  \cong tM\cong   M/tM \cong R/(x,t)$.   Setting $S=\frac{k[[r,s,t,x]]}{(r^2,rs,s^2)}$, we may consider a trivial quasi-deformation $R \rightarrow R \leftarrow S$. Note that $\Hom_R(N,M) \cong R/(x,t) \cong S/(x,t)$ and that $\pd_S(R)=2$. Thus, $\pd_S(\Hom_R(N,M))-\pd_S(R)=2-2=0$ and $\operatorname{CI-dim}_R(\Hom(N,M))=0$.

(4) By \cite[Proposition 3.6]{GJT}, the claim follows from the fact that $M$ and $N$ admit periodic resolutions.
\end{proof}

\begin{agra} We would like to thank David Jorgensen for many valuable discussions about this paper and for helping us with Example \ref{ex:nongorensteinSym}. The authors thank Kaito Kimura for valuable comments and for sharing the argument in Remark \ref{remark:GpdRHom} with us. The authors also wish to thank Lars Christensen for his comments on a first draft of this paper. This work was done while the second author was visiting David Jorgensen at the University of Texas at Arlington. He gratefully acknowledges Jorgensen's guidance and the many valuable discussions throughout his stay.
\end{agra}

\begin{funding}
The second author was supported by grants 2022/12114-0 and 2024/17809-1, São Paulo Research Foundation (FAPESP). The third author was supported by grant 2022/03372-5, São Paulo Research Foundation (FAPESP). 
\end{funding}


\begin{thebibliography}{PTW02}



\bibitem{araya} T. Araya. {\em A Homological Dimension Related to AB Rings}, Beitr. Algebra Geom., vol. 60, no. 2, June 2019, pp. 225–31.




\bibitem{AC1} M. Auslander, I. Reiten. {\em Selected Works of Maurice Auslander}, American Mathematical Society, 1999. Collected Works Series 10.




\bibitem{avramov} L. L. Avramov, R.-O. Buchweitz. {\em Support varieties and cohomology over complete intersections},  Invent. Math., vol. 142, no. 2, Nov. 2000, pp. 285–318.




\bibitem{BERGH} P. A. Bergh. {\em Modules with reducible complexity}, J. Algebra, vol. 310, no. 1, April 2007, pp. 132-147. 




\bibitem{be-jo} P. A. Bergh, D. A. Jorgensen. {\em On the Vanishing of Homology for Modules of Finite Complete Intersection Dimension}, J. Pure Appl. Algebra, vol. 215, no. 3, Mar. 2011, pp. 242–52.










\bibitem{LarsBook} L. W. Christensen, H.-B. Foxby, H. Holm. {\em Derived Category Methods in Commutative Algebra}, 1st ed, Springer, 2024. Springer Monographs in Mathematics Series.




\bibitem{CH} L. W. Christensen, H. Holm. {\em Algebras That Satisfy Auslander’s Condition on Vanishing of Cohomology}, Math. Z., vol. 265, no. 1, May 2010, pp. 21–40.




\bibitem{LarsHolmAC} L. W. Christensen, H. Holm. {\em Vanishing of Cohomology over Cohen–Macaulay Rings}, Manuscripta Math., vol. 139, nos. 3–4, Nov. 2012, pp. 535–44.




\bibitem{LarsDave} L. W. Christensen, D. A. Jorgensen. {\em Tate (Co)Homology via Pinched Complexes}, Trans. Amer. Math. Soc., vol. 366, no. 2, May 2013, pp. 667–89.




\bibitem{GJT} M. Gheibi, D. A. Jorgensen, R. Takahashi. {\em Quasi-Projective Dimension}, Pacific J. Math., vol. 312, no. 1, Aug. 2021, pp. 113–47.




\bibitem{Ghosh-Puth} D. Ghosh, T. J. Puthenpurakal. {\em Gorenstein Rings via Homological Dimensions, and Symmetry in Vanishing of Ext and Tate Cohomology}, Algebr. Represent. Theory, vol. 27, no. 1, Feb. 2024, pp. 639–53.



\bibitem{HO-VH-VD} R. Holanda, V. H. Jorge-Pérez, V. D. Mendoza-Rubio. {\em  On Extension modules of finite homological dimension}, arxiv preprint arXiv:2504.15224v2.




\bibitem{JorgensenHuneke} C. Huneke, D. A. Jorgensen. {\em Symmetry in the Vanishing of Ext over Gorenstein Rings}, Math. Scand., vol. 93, no. 2, Dec. 2003, pp. 161.








\bibitem{JO-LI} D. A. Jorgensen, L. M. Şega. {\em Nonvanishing Cohomology and Classes of Gorenstein Rings}, Adv. Math., vol. 188, no. 2, Nov. 2004, pp. 470–90.




\bibitem{JO-SE} D. A. Jorgensen, L. M. Şega. {\em Asymmetric complete resolutions and vanishing of ext over Gorenstein rings}, Int. Math. Res. Not. IMRN, vol. 2005, no. 56, Jan. 2005, pp. 3459-3477.




\bibitem{PE-JO} P. Jørgensen. {\em Symmetry Theorems for Ext Vanishing}, J. Algebra, vol. 301, no. 1, July 2006, pp. 224–39.





\bibitem{MMDN} P. Martins, V. D. Mendoza-Rubio, and Z. Nason. {\em Finiteness of complete intersection dimensions of RHom complexes and Ext modules}, arxiv preprint arXiv:2601.07811v2.




\bibitem{IGOR} I. J. Nascimento, V. H. Jorge-Pérez, T. H. Freitas. {\em Some Homological Conjectures over Idealization Rings}, Bull. Braz. Math. Soc. (N.S.), vol. 56, no. 4, Dec. 2025.




\bibitem{TakahashiUAC} S. Nasseh, S. Sather-Wagstaff, R. Takahashi, and K. VandeBogert. {\em Applications and Homological Properties of Local Rings with Decomposable Maximal Ideals}, J. Pure Appl. Algebra, vol. 223, no. 3, Mar. 2019, pp. 1272–87.




\bibitem{NASSEH-SYMMETRY} S. Nasseh, M. Tousi. {\em A Note on Symmetry in the Vanishing of Ext}, Rocky Mountain Journal of Mathematics, vol. 43, no. 1, Feb. 2013.




\bibitem{sanders} W. T. Sanders. {\em Categorical and homological aspects of module theory over commutative rings}, Thesis (Ph.D.) University of Kansas, 2015.






\bibitem{sather} S. Sather-Wagstaff. {\em Complete Intersection Dimensions for Complexes}, J. Pure Appl. Algebra, vol. 190, nos. 1–3, June 2004, pp. 267–90.








\bibitem{levins} A. J. Soto Levins. {\em The Auslander bound for complexes}, Arch. Math. (Basel), vol. 126. no. 1, Jan. 2026, pp. 29-39.





 


  






















\bibitem{veliche} O. Veliche. {\em Gorenstein Projective Dimension for Complexes}, Trans. Amer. Math.
Soc., vol. 358, no. 3, May 2005, pp. 1257–83.




\bibitem{WJ} J. Wei. {\em Auslander Bounds and Homological Conjectures}, Rev. Mat. Iberoam., vol. 27, no. 3, Dec. 2011, pp. 871–84.














\end{thebibliography}
\end{document}